\documentclass[12pt]{amsart}
\usepackage{color}
\usepackage{amstext}
\usepackage{amsthm}
\usepackage{amsrefs}
\usepackage{amsmath}
\usepackage{amssymb}
\usepackage{latexsym}
\usepackage{amsfonts}

\begin{document}

\newcommand{\norm}[1]{\ensuremath{\|#1\|}}
\newcommand{\abs}[1]{\ensuremath{\vert#1\vert}}
\newcommand{\p}{\ensuremath{\partial}}
\newcommand{\pr}{\mathcal{P}}

\newcommand{\pbar}{\ensuremath{\bar{\partial}}}
\newcommand{\db}{\overline\partial}
\newcommand{\D}{\mathbb{D}}
\newcommand{\B}{\mathbb{B}}
\newcommand{\Sp}{\mathbb{S}}
\newcommand{\T}{\mathbb{T}}
\newcommand{\R}{\mathbb{R}}
\newcommand{\Z}{\mathbb{Z}}
\newcommand{\C}{\mathbb{C}}
\newcommand{\N}{\mathbb{N}}
\newcommand{\scrH}{\mathcal{H}}
\newcommand{\scrL}{\mathcal{L}}
\newcommand{\td}{\widetilde\Delta}

\newcommand{\La}{\langle }
\newcommand{\Ra}{\rangle }
\newcommand{\rk}{\operatorname{rk}}
\newcommand{\card}{\operatorname{card}}
\newcommand{\ran}{\operatorname{Ran}}
\newcommand{\im}{\operatorname{Im}}
\newcommand{\re}{\operatorname{Re}}
\newcommand{\tr}{\operatorname{tr}}
\newcommand{\vf}{\varphi}
\newcommand{\f}[2]{\ensuremath{\frac{#1}{#2}}}

\numberwithin{equation}{section}

\newtheorem{thm}{Theorem}[section]
\newtheorem{lm}[thm]{Lemma}
\newtheorem{defi}[thm]{Definition}
\newtheorem{cor}[thm]{Corollary}
\newtheorem{conj}[thm]{Conjecture}
\newtheorem{prob}[thm]{Problem}
\newtheorem{prop}[thm]{Proposition}
\newtheorem*{prop*}{Proposition}
\newtheorem{lemma}[thm]{Lemma}

\theoremstyle{remark}
\newtheorem{rem}[thm]{Remark}
\newtheorem*{rem*}{Remark}

\title{Simultaneous Stabilization in $A_\R(\D)$}

\author[R. Mortini]{Raymond Mortini}
\address{Raymond Mortini, D\'{e}partement de Math\'{e}matiques\\ LMAM, UMR 7122, Universit\'{e} Paul Verlaine\\ Ile du Saulcy\\  F-57045 Metz, France}
\email{mortini@poncelet.univ-metz.fr}

\author[B. D. Wick]{Brett D. Wick$^*$}
\address{Brett D. Wick, Department of Mathematics\\ University of South Carolina\\ LeConte College\\ 1523 Greene Street\\ Columbia, SC USA 29208}
\address{Present:  The Fields Institute\\ 222 College Street, 2nd Floor\\ Toronto, Ontario\\ M5T 3J1 Canada}
\email{wick@math.sc.edu}
\thanks{$*$ Research supported in part by a National Science Foundation DMS Grant \# 0752703.}



\begin{abstract} 
In this note we study the problem of simultaneous stabilization for the algebra $A_\R(\D)$.  
Invertible pairs $(f_j,g_j)$, $j=1,\ldots, n$, in  a commutative unital algebra are called \textit{simultaneously stabilizable} if there exists a pair $(\alpha,\beta)$ of elements such that $\alpha f_j+\beta g_j$ is invertible in this algebra for $j=1,\ldots, n$.  
 
 For $n=2$, the simultaneous stabilization problem admits a positive solution for any data if and only if the Bass stable rank of the algebra is one.  Since $A_\R(\D)$ has stable rank two, we are faced here with a different situation.  When $n=2$, necessary and sufficient conditions are given so that we have simultaneous stability in $A_\R(\D)$.

For $n\geq 3$ we show that under these conditions simultaneous stabilization is not possible and further connect this result to the question of which pairs $(f,g)$ in $A_\R(\D)^2$ are totally reducible; that is, for which pairs do there exist two units $u$ and $v$ in
$A_\R(\D)$ such that $uf+vg=1$.  
\end{abstract}

\maketitle

\section*{Introduction}
\label{s1}

Given a commutative ring (or an algebra) $R$ with unit $1$, we say that a pair $(f,g)\in R^2$ is \textit{invertible} if there exists  $(\alpha,\beta)\in R^2$ such that
$$
\alpha f+\beta g=1,
$$
and write $(f,g)\in U_2(R)$.

We say that $n$ invertible elements $(f_j,g_j)\in U_2(R)$ are \textit{simultaneously stabilizable} if there exists $(\alpha, \beta)\in R^2$ such  that for $j=1,\ldots ,n$
$$
\alpha f_j+\beta g_j\in R^{-1},
$$
where $R^{-1}$ denotes the set of invertible elements in the ring $R$.  

When $n=2$ the notion of simultaneously stabilizable is very close to the notion of the ring $R$ having Bass stable rank one.  Since this notion will play a role in the proofs, we recall this now.  We say that the ring $R$ has \textit{Bass stable rank one} if for any invertible pair $(f,g)\in R^2$ there exists an $h\in R$ such that
$$
f+hg\in R^{-1}.
$$
Note that this can be rephrased as asking for the existence of $\alpha\in R^{-1}$ and $\beta\in R$ such that
$$
\alpha f+\beta g=1.
$$

A further property, even stronger than having Bass stable rank one, called the \textit{unit-1 stable rank}, is to want for every invertible pair $(f,g)\in U_2(R)$ the existence of $\alpha\in R^{-1}$ and 
$\beta\in R^{-1}$ such that
$$
\alpha f+\beta g=1.
$$

In the literature, a pair $(f,g)$ having this property is sometimes called \textit{totally reducible}.  This concept was introduced by P. Menal and J. Moncasi, see \cite{MM}. 
For rings of holomorphic functions on planar domains, this is no longer
of interest, since it is known that no such rings have unit-1 stable rank whenever they  properly
contain the constants ($\R$ or $\C$).  See \cite{MortRup}.
 A related concept is called the Godefroid-Goodearl-Menal property, see \cites{Gode,GM}.  This property is that for each $x,y\in R$ there exists a unit $u\in R^{-1}$ such that $x-u$ and $y-u^{-1}$ are invertible in $R$.  It is known that this property implies that $R$ has unit-1 stable rank.  

When these concepts are applied to various spaces of analytic functions many interesting questions arise.  For the disc algebra, $A(\D)$, these properties are well studied.  The notion of invertible $n$-tuples coincides here with the notion of $n$-tuples satisfying the Corona condition.  See \cite[p.~365]{Rudin}.
The stable rank of $A(\D)$ was computed by Jones, Marshall and Wolff, \cite{JMW}, and the concept of total reducibility and unit-1 stable rank  was studied by
Mortini and Rupp,\cite{MortRup},  and  Blondel, Mortini and Rupp,  \cite{BlonMortRup}.

But, as motivated by Control Theory, the disc algebra is not physically meaningful since the functions take complex values.  So one introduces a more useful algebra, and asks similar questions.

\subsection{Motivations and Main Results}

We will be interested in the case where $R$ is a certain ring of analytic functions, namely the real disc algebra $A_\R(\D)$.  The space $A_\R(\D)$ is the set of functions in the disc algebra with the additional property that
$$
f(z)=\overline{f(\overline{z})}\quad \forall z\in\overline{\D}.
$$
This definition is equivalent to the property that a function $f\in A_\R(\D)$ has a Fourier series expansion with real coefficients.

In this context the notion of invertibility is intimately connected with the Corona Theorem for these algebras.  A pair $(f,g)$ is invertible in $A_\R(\D)$ if and only if
$$
\abs{f(z)}+\abs{g(z)}\geq\delta>0\quad\forall z\in\overline{\D}.
$$ 

The necessity of this result is immediate, while the sufficiency follows from a symmetrization of the usual Corona Theorem for $A(\D)$.   Indeed, for functions which satisfy this condition, we can always find $\alpha,\beta\in A(\D)$ such that
$$
\alpha f+\beta g=1.
$$
See \cite{Rudin}.  One then defines
$$
\tilde{\alpha}(z):=\f{\alpha(z)+\overline{\alpha(\overline{z})}}{2}\quad\textnormal{ and }\quad \tilde{\beta}(z):=\f{\beta(z)+\overline{\beta(\overline{z})}}{2}.
$$
It is immediate that this is the solution to the Bezout equation in $A_\R(\D)$ that we seek.

The question of when the Bezout equation $\alpha f+\beta g=1$  associated with an invertible pair $(f,g)$ in $A_\R(\D)^2$ has a solution $(\alpha,\beta)\in A_\R(\D)^2$  with $\alpha^{-1}$ also in $A_\R(\D)$ was addressed in \cite{Wick1}.  It was shown
 that the pair $(f,g)$ must satisfy an additional condition which is both necessary and sufficient for the existence of $(\alpha,\beta)$ with the desired properties.  This condition will play a role in later arguments, so we recall the definition.

Given an invertible pair $(f,g)$ in $A_\R(\D)^2$, we will say that $f$ is of \textit{constant sign} on the real zeros of $g$, if $f$ has the same sign at all real zeros of $g$.  This condition arises naturally by examining what happens when you have a solution to the Corona problem with an invertible element.  

In fact, as was shown by the second author in \cite{Wick1}, if $(f,g)$ is an invertible pair in 
$A_\R(\D)$, then there exists $h\in A_\R(\D)$ such that $f+hg\in A_\R(\D)^{-1}$ if and only
if $f$ is of constant sign on the real zeros of $g$.  One calls a pair of functions which satisfy this property \textit{reducible}.

We are now going to see that if we have  a solution to a simultaneous stabilization problem in the real disc algebra, then we must have a similar additional necessary condition that our Corona data must satisfy.  Suppose that $(f_1,g_1)$ and $(f_2,g_2)$ are simultaneously stabilizable.  Then we can find functions $\alpha$ and $\beta\in A_\R(\D)$ such that 
$$
\alpha f_1+\beta g_1=1$$
and
$$
\alpha f_2+\beta g_2=u\in A_{\R}(\D)^{-1}.
$$

Using the matricial representation we get
$$\left(
\begin{array}{cc}
  f_1 & g_1 \\
 f_2 & g_2 \\   
\end{array}
\right) \left(
\begin{array}{c}
\alpha\\ \beta
\end{array}
\right)=\left(
\begin{array}{c}
1\\ u
\end{array}
\right).
$$

Then we see that at points $x\in [-1,1]$ where the determinant of the above matrix is zero, 
we have that $(f_2(x),g_2(x))=\lambda(x) (f_1(x),g_1(x))$ for some $\lambda(x)\in \R$.
Hence $\lambda(x)=u(x)$. Since $u$ is invertible, it has constant sign on $[-1,1]$.
Hence $\lambda(x)$ has the same sign at all the real zeros of the function $f_1g_2-f_2g_1$.

\begin {defi}
We say that the pairs $(f_1,g_1)$ and $(f_2,g_2)$ are ``\textit{sign-linked}''
if whenever $(f_2(x),g_2(x))=\lambda(x) (f_1(x),g_1(x))$ for some $x\in [-1,1]$ and $\lambda(x)\in \R$ the function $\lambda(x)$ has constant sign  on the set of real singular points
of the matrix  $\left(
\begin{array}{cc}
  f_1 & g_1 \\
 f_2 & g_2 \\   
\end{array}
\right) $.
\end{defi}
Note that for invertible pairs $(f_j,g_j)$ this notion is symmetric, since $\lambda(x)\not=0$.
We also observe that this is a reasonable (and correct) generalization of the concept being positive on the real zeros of a function.  When $(f_1,g_1)=(1,0)$ and $(f_2,g_2)=(f,g)$ then these pairs 
are sign-linked if and only if $f$ has constant sign on the real zeros of $g$.

One can also ask for more in terms of the solution to the Corona problem.  For example, we are interested in pairs $(f,g)$ of functions in $A_\R(\D)$ that are totally reducible.  This is motivated by the fact  that the ring $A_\R(\D)$ fails to have the unit-1 stable property, since the invertible pair $(z, 1-z^2)$ is not even reducible.

In the context  of $A_\R(\D)$, it is important to note that if 
a pair is totally reducible, then the Corona data must satisfy an additional necessary property. 
To see this, suppose that for $(f,g)\in A_\R(\D)^2$ it is possible to find
 $u, v\in A_\R(\D)^{-1}$ such that
$$
u f+v g=1.
$$
Then $f$ has constant sign on the real zeros of $g$ and similarly $g$ 
must have constant sign on the real zeros of $f$.  This condition on the zeros of $f$ and $g$ is typically called the even interlacing property in the Control Theory literature.  The counterexample that will be constructed will have this necessary property as well.  

\subsubsection{Main Results}

\begin{thm}
\label{SimulStab}
Invertible pairs $(f_1,g_1)$ and $(f_2, g_2)$ of functions in the algebra $A_\R(\D)$ are simultaneously stabilizable if and only if they are sign-linked.


\end{thm}

We next show that when we consider more than two pairs of Corona data, any two of them being sign-linked, then they are generally not simultaneously stabilizable.  Of course we must add here the sign-linked condition, since otherwise we would already have
a counterexample for the case of two pairs.

It is enough to show this in the case of three pairs of functions.  The construction is similar to what was done in \cite{BlonMortRup}.

\begin{thm}
\label{SimulStab3+}
There exist three pairs of functions $(f_j,g_j)\in U_2(A_\R(\D))$, with $\{(f_1,g_1), (f_2,g_2)\}$, 
$\{(f_1,g_1), (f_3,g_3)\}$ and $\{(f_2,g_2), (f_3,g_3)\}$ being sign-linked,
 that are not simultaneously stabilizable.  That is, for $j=1,2,3$, the problem 
$$
\alpha f_j+\beta g_j \in A_\R(\D)^{-1},
$$
has no solution with $(\alpha,\beta)\in A_\R(\D)^2$.
\end{thm}

As a Corollary to this Theorem, we have the following result which says that the ring $A_\R(\D)$ does not have the unit-1 stable property,
\begin{cor}
\label{FailureofEIP}
There exists a pair $(f,g)\in U_2(A_\R(\D))$ with $f$ being positive on the real zeros of $g$ and $g$ positive on the real zeros of $f$, such that  if 
$$
\alpha f+\beta g=1,
$$
then either $\alpha$ or $\beta$ is not invertible in $A_\R(\D)$. 
\end{cor}

We observe at this point, that if one wants only one of $\alpha$ or $\beta$ invertible, then this is possible and can be found in \cite{Wick1}.

\begin{rem*}
Many of these results have interpretations and motivations in Control Theory.  The interested reader can see these connections in the book by V. Blondel, see \cite{Blon}, which is an excellent reference for the motivations of the problems of simultaneous stabilization in Control Theory.  Additionally, the book by Vidyasagar \cite{Vidy} is a good introduction to Control Theory and connections to the Bezout equation.
\end{rem*}

\section {Some general facts on invertible n-tuples}

Let $R$ be a commutative unital ring, with the unit being denoted by $1$.
We begin with some easy facts on invertible $n$-tuples  and the representations of  the unit element 
of the ring generated by these $n$-tuples.
Denote by $U_n(R):=\{(x_1,\dots,x_n)\in R^n: \exists y_j\in R: \sum_{j=1}^n x_jy_j=1\}$
 the set of invertible $n$-tuples in $R^n$. Finally, for $x$ and $f\in R^n$,  let $\left<x,f\right> \;:=\;
 \sum_{j=1}^n x_jf_j$. 

\begin{lemma}\label{inv}
 Let $f,g\in R^n$ and let $M$ be an $n\times n$-matrix over $R$.  Suppose that $g=Mf$
 and that $g\in U_n(R)$. Then  $f\in U_n(R)$. In particular, if $M$ is invertible
 and $f\in U_n(R)$, then $g\in U_n(R)$, too.
\end{lemma}

\begin{proof}
By hypothesis, $1=\left<g, a\right>$ for some $a\in R^n$.  Hence

$$1=\left<Mf, a\right>= \left<f, M^\bot a\right>.$$

\end{proof}

\begin{prop}\label{gen}
Suppose that $(f_1,\ldots,f_n)$ is an invertible n-tuple in $R^n$ and let $1=\sum_{j=1}^n x_j f_j=\left<x,f\right>$.
 Then every other representation $1=\sum_{j=1}^n y_jf_j =\left<y,f\right>$ of $1$ can be deduced form
 the former by letting $y=x+H f$, where $H$ is an antisymmetric $n\times n$- matrix over $R$;
 that is $H=-H^\bot$, where $H^\bot$ is the transpose of $H$.
\end{prop}
\begin{proof}
Suppose that  $1=\left<x,f\right>$ and $1=\left<y,f\right>$. Multiply these equations by $y_k$, respectively
$x_k$.  Then $x_k-y_k= \sum_{j\not=k} f_j(y_jx_k-y_kx_j)$. Thus $y=  x+ H f$ for some
antisymmetric matrix $H$.

The converse is easy, too. In fact suppose that $1=\left<x,f\right>$. Then

$$\left<y,f\right> = \left< x+H f, f\right>= \left< x,f\right> + \left< H f, f\right>=1+0=1$$
because
$$
\left< H f, f\right>= \left< f, H^\bot f\right>= -\left<f, H f\right>=-\left<Hf, f\right>.
$$
\end{proof}

The following result is mentioned in the unpublished manuscript \cite{BlonMortRup}.
The proof works on the same lines as that of our Theorem \ref{SimulStab}.

\begin{thm} \label{bass}
Let $R$ be a commutative unital ring. Then every simultaneous stabilization 
problem $\alpha f_j +\beta g_j \in R^{-1}, ~~~ j=1,2$,  with $(f_j,g_j)\in U_2(R)$  is solvable if and only if 
$R$ has Bass stable rank one.
\end{thm}
\begin{proof}
Let us assume that $R$ has Bass stable rank one. Suppose that $\alpha f_1+\beta g_1=1$. 
By Lemma \ref{gen}, every other representation of the unit element (with generators $(f_1, g_1)$)
has the form
$$1= (\alpha+hg_1)f_1+(\beta-hf_1)g_1$$
for some $h\in R$. Consider now the element
$$u:=(\alpha+hg_1)f_2 +(\beta-hf_1)g_2;$$
which, after algebra reduces to 
$$u= (\alpha f_2+\beta g_2)+ h( g_1f_2-f_1g_2).$$

Let $F=\alpha f_2+\beta g_2$ and $G=g_1f_2-f_1g_2$.  One observes that the pair $(F,G)$ can be written in matrix notation as
\begin{displaymath}
\left(
\begin{array}{c}
  F   \\
  G   \\   
\end{array}
\right)=\left(
\begin{array}{cc}
  \alpha & \beta \\
  g_1 & -f_1 \\   
\end{array}
\right)
\left(
\begin{array}{c}
  f_2 \\
  g_2 \\   
\end{array}
\right).
\end{displaymath}

The corresponding $2\times 2$ matrix has determinant $-1$ and since the pair $(f_2,g_2)$ is invertible then the pair $(F,G)$ is invertible, by Lemma \ref{inv}.   But, by our assumption, $R$ has Bass stable rank one, and hence there exists an element $h\in R$ such that 
$$
u= F+hG\in R^{-1}.
$$
 To show the reciprocal, we just have to note that the simultaneous stabilization of the system
 $(1,0)$ and $(f,g)$ is nothing but the existence of an invertible element $\alpha$ and
 some $\beta$ so that $\alpha f +\beta g\in R^{-1}$.
\end{proof}

\section{Proofs of Main Results}
\label{s2}

Whereas by Theorem \ref{bass} each $2\times2$ problem 
$$\alpha f_j +\beta g_j \in A(\D)^{-1}, ~~ j=1,2$$
 with Corona data in the disc algebra
$A(\D)$ is solvable (since $A(\D)$ has stable rank one, see \cite{JMW}), the situation differs in $A_\R(\D)$.
We can not apply the results of Theorem \ref{bass}, since by a result of Rupp and Sasane, \cite{RS1}, we have that the Bass stable rank of $A_\R(\D)$ is two.  Thus, we will have to impose additional conditions on the Corona data that we consider so that solutions will exist.



\begin{proof}[Proof of Theorem \ref{SimulStab}]

Suppose that
$$
\abs{f_1(z)}+\abs{g_1(z)}\geq\delta\quad\forall z\in\overline{\D}.
$$
By the Corona Theorem for $A_\R(\D)$ there exists $(\alpha,\beta)\in A_\R(\D)^2$ such 
that  
$$
\alpha f_1+\beta g_1=1.
$$

By Lemma \ref{gen}, every other representation of the unit element (with generators $(f_1, g_1)$)
has the form
$$1= (\alpha+hg_1)f_1+(\beta-hf_1)g_1$$
for some $h\in A_\R(\D)$. Consider now the function
$$u:=(\alpha+hg_1)f_2 +(\beta-hf_1)g_2;$$

that is 
$$u= (\alpha f_2+\beta g_2)+ h( g_1f_2-f_1g_2).$$
By  \cite{Wick1}, there exists $h\in A_\R(\D)$ such that $u$ is invertible if (and only if)
$F:=\alpha f_2+\beta g_2$ has constant sign on the real  zeros of $G:=g_1f_2-f_1g_2$.
We will show that our hypothesis, that $(f_1,g_1)$ and $(f_2,g_2)$ are sign-linked,
guarantees this property. In fact let  $G(x)=0$, where $-1\leq x\leq 1$. Then
$x$ is a critical point of the matrix
$\left(
\begin{array}{cc}
  f_1 & g_1 \\
 f_2 & g_2 \\   
\end{array}
\right) $. Hence   $(f_2(x), g_2(x))=\lambda(x)(f_1(x), g_1(x))$ for some $\lambda(x)\not=0$. 
So
$$F(x)=\alpha(x) f_2(x)+\beta(x)g_2(x)=\lambda(x)\left( \alpha(x) f_1(x)+\beta(x)g_1(x)\right)
=\lambda(x).$$

Our assumption implies that the sign of these values for $\lambda(x)$ does not vary with $x$. Hence,
$F$ has constant sign on the zeros of $G$. Thus, there is a joint solution $(\tilde\alpha,\tilde\beta)=(\alpha+hg_1,\beta-hf_1)$ to our problem 
$$\tilde \alpha f_j +\tilde \beta g_j\in A_\R(\D)^{-1}, ~~ j=1,2.$$
\end{proof}

\begin{rem}
We have the following examples of pairs of functions for which the simultaneous stabilization problem is solvable:

\begin{enumerate}
\item Let $(f_1,g_1)=(1,0)$ and $(f_2,g_2)=(f,g)$, where $(f,g)$ is any invertible pair in $A_\R(\D)$ such that $f>0$ on the real zeros of $g$.\\

\item Let  $(f_j,g_j)=(f,g)$, $(j=1,2)$, where $(f,g)\in U_2(A_\R(\D))$ is arbitrary.\\

\item  Let  $(f,g)\in U_2(A_\R(\D))$, $(f_1,g_1)=(f,g)$ and $(f_2,g_2)=(g,f)$, and suppose
that $f$ avoids $g$ on $[-1,1]$; that is $f(x)\not=g(x)$ for any $x\in [-1,1]$.
Then the system $$
 \left\{
 \begin{array}{ccc}
\alpha f +\beta g & \in & A_{\R}(\D)^{-1}\\
\alpha g+\beta f  & \in & A_{\R}(\D)^{-1}
\end{array}
\right.$$
is solvable in $A_\R(\D)$.
\end{enumerate}
\end{rem}

We want to point out the following classes of simultaneous stabilization problems:
\begin{prop}\label{exa}
Let $(f_1,g_1)$ and $(f_2,g_2)$ be Corona data in $A_\R(\D)^2$. Then
the system 
$$\alpha f_j^2 +\beta g_j\in A_\R(\D)^{-1},~~ (j=1,2)$$
is solvable. 
\end{prop}
\begin{proof}
Assume that $x$ is a critical point of the matrix

$$ A=\left(
\begin{array}{cc}
  f_1^2 & g_1 \\
 f_2^2 & g_2 \\   
\end{array}
\right).$$
The vector $(f_2^2(x), g_2(x))$ is a nonzero multiple of the vector  $(f_1^2(x), g_1(x))$,
say $(f_2^2(x), g_2(x))=\lambda(x) \;(f_1^2(x), g_1(x))$. This obviously implies that  $\lambda(x)>0$.
Hence $(f_2^2(x), g_2(x))$ and $(f_2^2(x), g_2(x))$ are sign-linked. Now use Theorem
\ref{SimulStab} to get the solution.
\end{proof}

\begin{rem}
We note that whenever $F_1$ and $F_2$ are outer functions in $A_\R(\D)$,  then
every system 
 $$\left\{
 \begin{array}{ccc}
 
\alpha F_1 +\beta g_1 & \in & A_{\R}(\D)^{-1}\\
\alpha F_2+\beta g_2  & \in & A_{\R}(\D)^{-1}
\end{array}
\right.
$$
of Corona data   is solvable. This follows from Proposition \ref{exa} above and the fact
that outer functions $F\in A_\R(\D)$ with $F(0)>0$ have a square root
in  $A_\R(\D)$.
\end{rem}

We shall now prove Theorem \ref{SimulStab3+} which deals with the simultaneous stabilization
problem of three pairs of data.

\begin{proof}[Proof of Theorem \ref{SimulStab3+}]
The construction of this counterexample is very similar to the one constructed in \cite{BlonMortRup}.  Since we are after a little more (namely Corollary \ref{FailureofEIP}), we modify that construction, but, remark that it is possible to use their examples immediately to prove Theorem \ref{SimulStab3+} without the desire to have a sign-linked counterexample. 

Choose the following invertible pairs,
$$
(f_1,g_1)=(1,0)\quad (f_2,g_2)=(1,z^2)\quad (f_3,g_3)=(n^2z^2,1).
$$
It is immediate these three pairs are invertible.  But, we must show that 
$\{(f_1,g_1), (f_2,g_2)\}$, $\{(f_1,g_1), (f_3,g_3)\}$ and $\{(f_2,g_2), (f_3,g_3)\}$ are sign-linked.  Since $f_2$ is positive on the real zeros of $g_2$ the pair  $\{(f_1,g_1), (f_2,g_2)\}$ is sign-linked.  An identical statement holds for the pair $\{(f_1,g_1), (f_3,g_3)\}$.  It only remains to address why $\{(f_2,g_2), (f_3,g_3)\}$ is sign-linked.  First, a simple computation shows that the matrix corresponding to the pair $\{(f_2,g_2), (f_3,g_3)\}$ has real singular values of $\pm\frac{1}{\sqrt{n}}$. If we let $\lambda(x)=n$ when $x=\pm\frac{1}{\sqrt{n}}$, then $(f_3(x),g_3(x))=\lambda(x)(f_2(x),g_2(x))$. So the pair $\{(f_2,g_2), (f_3,g_3)\}$ is sign-linked.

Suppose that every triple of pairs were simultaneously stabilizable, then for every integer $n\in\N$ there exist $\alpha_n,\beta_n\in A_{\R}(\D)$ such that 
\begin{eqnarray*}
\alpha_n & \in & A_{\R}(\D)^{-1}\\
\alpha_n+\beta_nz^2 & \in & A_{\R}(\D)^{-1}\\
n^2\alpha_n z^2+\beta_n & \in & A_{\R}(\D)^{-1}.
\end{eqnarray*}

One then rewrites this as a system of two equations, since the first equation is just the assumption that $\alpha_n$ is invertible.  Doing so we have
\begin{eqnarray*}
1+h_nz^2 & \in & A_{\R}(\D)^{-1}\\
n^2z^2+h_n & \in & A_{\R}(\D)^{-1}.
\end{eqnarray*}
With this in hand, define the following auxiliary function,
$$
\varphi_n(z):=\frac{n^2z^4+h_n(z)z^2}{1+h_n(z)z^2}=z^2\frac{n^2z^2+h_n(z)}{1+h_n(z)z^2}.
$$
These functions are analytic and have no zeros in $\D\setminus\{0\}$.  Additionally, the function $\varphi_n$ attains the value $w=1$ only four times in $\D$, at the points $\pm\f{1}{\sqrt{n}},\pm\f{i}{\sqrt{n}}$.  By the generalized Montel's normal family criterion, the family of functions $\varphi_n$ is normal in $\D\setminus\{0\}$.  Without loss of generality, we may assume that $\varphi_n$ converges uniformly on compact subsets of $\D\setminus\{0\}$.  Then there are only two cases.

Case 1:  $\varphi_n$ tends locally uniformly to infinity, i.e., the function $\varphi_n^{-1}$ tends locally uniformly to $0$.

For $\epsilon>0$ we have
$$
\left\vert\frac{1+h_n(z)z^2}{n^2z^2+h_n(z)}\right\vert\leq\epsilon\abs{z}^2\quad n\geq N(\epsilon),\, \abs{z}=\frac{1}{2}.
$$
Let $\psi_n(z):=\frac{1+h_n(z)z^2}{n^2z^2+h_n(z)}$. Then for $n$ sufficiently large we have that 
$$
\abs{\psi_n(z)}\leq \frac{1}{8},\quad \abs{z}=\f{1}{2}.
$$

But, note that simple algebra shows that
$$
\f{h_n(z)}{n^2}\left[\psi_n(z)-z^2\right]=\f{1}{n^2}-z^2\psi_n(z).
$$
Using this, we see that for all integers sufficiently large
$$
\left\vert\f{h_n(z)}{n^2}\right\vert\leq\f{\f{1}{n^2}+\f{\epsilon}{2}}{\f{1}{4}-\f{1}{8}},\quad \abs{z}=\f{1}{2}.
$$
 The maximum modulus principle implies the same inequality for all $z$ such that $\abs{z}\leq\frac{1}{2}$.  From above we know that all the functions $u_n(z):=z^2+\f{1}{n^2}h_n(z)$ are invertible in $\D$.  But, these functions tend uniformly to the function $z^2$ in $\abs{z}\leq\f{1}{2}$, which is neither invertible nor identically zero.  This contradicts Hurwitz's Theorem, and so this case is impossible.

Case 2: $\varphi_n$ tends locally uniformly to an analytic function $\varphi$ in $\D\setminus\{0\}$.

In this case we have that the functions $\varphi_n$ are uniformly bounded on compact subsets of $\D\setminus\{0\}$, say,
$$
\abs{\varphi_n(z)}\leq M\quad n\in\N,\, \abs{z}=\f{1}{2}.
$$

We additionally have that
$$
\varphi_n(z)=1+\frac{n^2z^4-1}{1+h_n(z)z^2},
$$
which implies that
$$
\left\vert\frac{n^2z^4-1}{1+h_n(z)z^2}\right\vert\leq M+1\quad n\in\N,\, \abs{z}=\f{1}{2}.
$$
But, for $n$ large we see that the following inequality must hold,

$$
\left\vert\frac{n^2}{1+h_n(z)z^2}\right\vert\leq\f{M+1}{\f{1}{16}-\f{1}{n^2}}.
$$
The maximum modulus principle implies the same inequality for all $z$ such that $\abs{z}\leq\f{1}{2}$.  Evaluating this inequality when $z=0$  we obtain a condition which is obviously false for large $n$, i.e.,
$$
n^2\leq\f{M+1}{\f{1}{16}-\f{1}{n^2}}.
$$

To sum up,  for all $n$ large we have shown that it is impossible for the systems given above to be simultaneously stabilizable. So we are done.
\end{proof}

Using Theorem \ref{SimulStab3+} we show that it is in general impossible for there to exist solutions to the Bezout equation in $A_\R(\D)$ that are both invertible.  This addresses Corollary \ref{FailureofEIP}.

\begin{proof}[Proof of Corollary \ref{FailureofEIP}]
The proof is by contradiction.  Suppose that for every invertible pair $(f,g)$ in $A_\R(\D)^2$ with $f$ positive on the real zeros of $g$ and $g$ positive on the real zeros of $f$ one could find invertible elements $\alpha$ and $\beta$ in $A_\R(\D)$ such that
$$
1=\alpha f+\beta g.
$$

Consider the following functions $f(z)=z^2$ and $g(z)=1-n^2z^4$.  Then it is trivial that the pair $(f,g)$ is invertible, $f$ is positive on the real zeros of $g$, and $g$ is positive on the real zeros of $f$.  Thus, there exist invertible functions $u_n$ and $v_n$ in $A_\R(\D)$ such that 
$$
v_n=u_nf+g.
$$
Hence
$$
v_n(z)=(u_n(z)-n^2z^2+n^2z^2)z^2+1-n^2z^4=(u_n(z)-n^2z^2)z^2+1.
$$
Now let $h_n(z):=u_n(z)-n^2z^2$. Then we obtain that
\begin{eqnarray*}
1+h_nz^2 & \in & A_{\R}(\D)^{-1}\\
n^2z^2+h_n & \in & A_{\R}(\D)^{-1}.
\end{eqnarray*}
But, we know from the proof of Theorem \ref{SimulStab3+} that this is impossible for all integers $n$.  The desired counterexample then follows by taking $n$ sufficiently large.
\end{proof}

\section{Totally Reducible pairs in $A_\R(\D)$}

Recall that a pair $(f,g)$ in $A_\R(\D)^2$ is said to be totally reducible if there exist $u,v$ invertible
in $A_\R(\D)$ so that $uf+vg=1$. Corollary \ref{FailureofEIP} above, for example, showed that the pair 
$(z^2, 1-n^2z^4)$ is not totally reducible. On the other hand, it is easy to see that 
the pair $(f,g)$ is totally reducible if and only if the system $(1,0), (0,1), (f,g)$ of
three invertible pairs in $A_\R(\D)^2$ is simultaneous stabilizable.  We shall now show that
 large classes of pairs  are totally reducible.
The following is an analogue of Lemma 4 in \cite{MortRup}.
\begin{lm}
Let $f\in A_\R(\D)$ be so that there exists $x_n\in \R\setminus f(\overline{\D})$  with  $x_n\to 0$.
Then for every $g\in A_\R(\D)$ such that $(f,g)$ is an invertible pair and such that 
$g$ has constant sign on the real zeros of $f$  there exist two invertible functions
 $u$ and $v$ in $A_\R(\D)$ such that $uf+vg=1$.
\end{lm}

\begin{proof}
 Let $g\in A_\R(\D)$ be such that $(f,g)$ is an invertible pair.  Since $g$ is assumed to have constant
  sign on the real zeros of $f$, there exist, by \cite{Wick1},  a function $h\in A_\R(\D)$ and
    a unit  $v\in A_\R(\D)^{-1}$ such that 
 \begin{equation}\label{redu}
  hf +vg =1.
  \end{equation}
Choose $M>0$ large enough so that $f-M$ is invertible in $A_\R(\D)$; e.g. let $M=||f||_\infty+1$. 
Multiplying (\ref{redu}) by a real number $\epsilon$ to be specified later and adding
on both sides $\frac{f}{f-M}$ yields the following equation
\begin{equation}\label{toredu}
\epsilon vg + \left( \frac{1}{f-M} +\epsilon h \right)f =\epsilon +\frac{f}{f-M} =(\epsilon+1)\frac{1}{f-M}
\left( f-\frac{\epsilon M}{\epsilon+1}\right).
  \end{equation}
  Since $x_n\to 0$ and $x_n\in\R$, we may choose $\epsilon_n\in \R$ so that $x_n=\frac{\epsilon_n M}{1+\epsilon_n}$ and
  $|\epsilon_n|\leq \frac{1}{||f-M||_\infty\;||h||_\infty}$. Then the functions
$$
\frac{1}{f-M}+\epsilon_n h  \textnormal{ and }  (\epsilon_n+1)\frac{1}{f-M}\left( f-\frac{\epsilon_n M}{\epsilon_n+1}\right)
$$
  are invertible in $A_\R(\D)$. Using (\ref{toredu}) we conclude that $(f,g)$ is totally reducible.
  \end{proof}

\begin{rem*} 
Note that  the condition on $f$ 
implies that $f$ has constant sign on $]-1,1[$; hence on the real zeros of $g$. 
In fact, since $0$ is a boundary point of the image  of $f$, $f(\D)$ open implies that
$f$ cannot have any zero inside $\D$. Now use the intermediate value theorem on $[-1,1]$.
\end{rem*}

\begin{thm}
Let $f$ be an outer function in $A_\R(\D)$.   Then  for every $g\in A_\R(\D)$ such that $(f,g)$ is an invertible pair and such that 
$g$ has constant sign on the real zeros of $f$  there exist
 two invertible functions
 $u$ and $v$ in $A_\R(\D)$ such that $uf+vg=1$.
\end{thm}

\begin{rem*}
Note that the assumption that $g$ has constant sign on the real zeros of $f$
is equivalent here to the hypothesis that $g(-1)g(1)>0$ whenever $f(-1)=f(1)=0$.
\end{rem*}

\begin{proof}
This works exactly in the same manner as that of the disc algebra case in \cite{MortRup}. We have
just to note that if $E$ is the zero set of an outer function  $f$ in $A_\R(\D)$, then $E$ is symmetric
with respect to the real axis; hence if $p_E$ is any peak function in $A(\D)$ associated with $E$,
then the function $q_E(z)=p_E(z)\overline{p_E(\overline{z})}$ is a peak function for $E$
that is in $A_\R(\D)$.
\end{proof}

\section{Concluding Remarks}

 Given what has been shown about the problem of simultaneous stabilization in $A_\R(\D)$, we propose two problems.

\begin{prob}
Give a complete description of those pairs $(f_j,g_j)$, $j=1,\dots,n$ of 
Corona data for which the $n\geq 3$ simultaneous stabilization problems are solvable in $A(\D)$ or $A_\R(\D)$.
\end{prob}

We remark here that this is a well known and extremely challenging problem in the Control Theory literature.  For example, it is known that condition on the real axis alone (parity interlacing, sign-linked, etc.) do not suffice to solve this problem.  See \cite{BGMR}.  We also note that when restricting to rational data, it is known that this problem is rationally undecidable.  See \cite{Blon} and \cite{BSVW} and the references there in for more information concerning what is known.

\begin{prob}
Give a characterization of those pairs of  functions $(f,g)$  in $A(\D)^2$ or $A_\R(\D)^2$
for which $(f,g)$ is totally reducible.
\end{prob}

\end{document}